\documentclass[11pt,reqno,a4paper]{amsart}

\usepackage[T1]{fontenc}
\usepackage[utf8]{inputenc}

\usepackage{lmodern}
\usepackage{microtype}
\usepackage{mathtools}
\usepackage{amssymb}
\usepackage{mathrsfs}
\usepackage{bm}
\usepackage{enumitem}
\usepackage{xcolor}
\usepackage{aliascnt}
\usepackage{hyperref}
\usepackage[nameinlink,capitalise,noabbrev]{cleveref}

\hypersetup{
  colorlinks=true,
  linkcolor=blue!45!black,
  citecolor=green!35!black,
  urlcolor=blue!55!black,
  pdftitle={Elastic BB knots are multifold circles}
}

\numberwithin{equation}{section}

\newtheorem{theorem}{Theorem}[section]
\newaliascnt{proposition}{theorem}
\newtheorem{proposition}[proposition]{Proposition}
\aliascntresetthe{proposition}
\newaliascnt{lemma}{theorem}
\newtheorem{lemma}[lemma]{Lemma}
\aliascntresetthe{lemma}
\newaliascnt{corollary}{theorem}
\newtheorem{corollary}[corollary]{Corollary}
\aliascntresetthe{corollary}
\newaliascnt{assumption}{theorem}
\newtheorem{assumption}[assumption]{Assumption}
\aliascntresetthe{assumption}
\newaliascnt{definition}{theorem}
\newtheorem{definition}[definition]{Definition}
\aliascntresetthe{definition}
\newaliascnt{remark}{theorem}
\newtheorem{remark}[remark]{Remark}
\aliascntresetthe{remark}

\crefname{proposition}{Proposition}{Propositions}
\crefname{lemma}{Lemma}{Lemmas}
\crefname{corollary}{Corollary}{Corollaries}
\crefname{assumption}{Assumption}{Assumptions}
\crefname{definition}{Definition}{Definitions}
\crefname{remark}{Remark}{Remarks}

\newcommand{\T}{\mathbb T}
\newcommand{\R}{\mathbb R}
\newcommand{\K}{\mathfrak K}
\newcommand{\Cclass}{\mathscr C}
\newcommand{\Length}{\mathscr L}
\newcommand{\TC}{\operatorname{TC}}
\newcommand{\Thi}{\operatorname{Thi}}
\newcommand{\Rop}{\operatorname{Rop}}
\newcommand{\Eb}{E_{\mathrm{bend}}}
\newcommand{\Fbar}{\overline E}
\newcommand{\MinRad}{\operatorname{MinRad}}
\newcommand{\dcsd}{\operatorname{dcsd}}

\newcommand{\weakto}{\rightharpoonup}
\newcommand{\sign}{\operatorname{sign}}

\title{Elastic BB knots are multifold circles}
\author{Philipp Reiter}
\address[Philipp Reiter]{Chemnitz University of Technology,
Faculty of Mathematics, 09107~Chemnitz, Germany}
\email{reiter@math.tu-chemnitz.de}	

\author{Heiko von der Mosel}
\address[Heiko von der Mosel]{RWTH Aachen University,
Institute for Mathematics,
Templergraben~55,
52062~Aachen, Germany}
\email{heiko@instmath.rwth-aachen.de}
\date{September 18, 2026}


\begin{document}

\begin{abstract}
We study elastic knots obtained by minimizing the bending energy of a unit
loop together with a small multiple of ropelength.  We assume a generalized
F\'ary--Milnor inequality on the $C^1$-boundary of a knot class: every
$H^2$-limit curve in the $C^1$-closure has total curvature at least $2\pi$
times the bridge index.  Under this assumption we prove that every elastic
knot in a one-component BB knot class, that is, a class whose bridge and
braid indices coincide, is the round circle covered precisely that many
times.

The proof does not differentiate the nonsmooth ropelength functional.  A
minimal closed braid is contracted in the fibres of a solid torus, producing
a {family of comparison curves}
with quadratic bending-energy excess and thickness of
linear order.  This yields a lower thickness bound for the regularized
minimizers.  {Variations} on an intermediate scale then give
free stationarity of the limiting curve for the normalized bending energy.
Constant curvature and free stationarity reduce the classification to a
linear distributional ordinary differential equation.
\end{abstract}

\maketitle

\section{Introduction}

For an arclength-parametrized closed curve
$\gamma\colon\T:=\R/\mathbb Z\to\R^3$%
, the bending energy is
\begin{equation}\label{eq:intro-bending}
  \Eb(\gamma):=\int_{\T}|\gamma''(s)|^2\,ds.
\end{equation}
Minimization of \eqref{eq:intro-bending} within a nontrivial knot class is
not a closed variational problem: minimizing sequences may develop
self-contact and leave the class of embeddings.  Following
\cite{GerlachReiterVonderMosel2017}, we use ropelength $\Rop(\cdot)$ 
as a self-avoidance
penalty and minimize
\begin{equation}\label{eq:intro-total}
  E_\vartheta(\gamma)
  :=\Eb(\gamma)+\vartheta\Rop(\gamma),
  \qquad \vartheta>0,
\end{equation}
among unit-length representatives of a prescribed tame knot class $\K$.
Subsequential limits as $\vartheta\downarrow0$ are called elastic knots.

The bridge index $\beta(\K)$ gives the lower bound
\begin{equation}\label{eq:bridge-lower-intro}
  \TC(\gamma)\ge 2\pi\beta(\K)
\end{equation}
for embedded representatives $\gamma$ of $\K$; 
this is contained in Milnor's crookedness
interpretation of total curvature \cite{Milnor1950}.  By Cauchy--Schwarz,
\eqref{eq:bridge-lower-intro} implies
\begin{equation}\label{eq:bending-lower-intro}
  \Eb(\gamma)\ge (2\pi\beta(\K))^2
\end{equation}
for unit loops.  Equality in this estimate forces constant curvature, but
constant curvature alone does not classify a nonembedded $H^2$-curve.  In
particular, the curvature vector may change direction at parameter values
which are mapped to a self-contact point.

The second topological quantity relevant here is the braid index
$\operatorname{braid}(\K)$.  Diao, Ernst, and Reiter
\cite{DiaoErnstReiter2021} call $\K$ a \emph{BB knot class} if
\begin{equation}\label{eq:BB}
  \operatorname{braid}(\K)=\beta(\K).
\end{equation}
The equality suggests that a minimal closed braid can collapse to the
$\beta(\K)$-covered circle while saturating
\eqref{eq:bending-lower-intro}.  The purpose of this paper is to turn that
picture into a proof, conditional on the following extension of
\eqref{eq:bridge-lower-intro} to the boundary of the knot class.

\begin{assumption}[Generalized F\'ary--Milnor inequality]\label{ass:GFM}
Let $\K$ be a one-component tame knot class and let
$\beta=\beta(\K)$.  Every unit-speed curve
$u\in H^2(\T,\R^3)$ which belongs to the $C^1$-closure of the unit loops in
$\K$ satisfies
\begin{equation}\label{eq:GFM}
  \TC(u)\ge 2\pi\beta.
\end{equation}
\end{assumption}

For $\beta=2$, an extension of the ordinary F\'ary--Milnor inequality to
the $C^1$-closure was proved in \cite[Theorem~A.1]{GerlachReiterVonderMosel2017};
see also \cite[Theorem~4.4]{GilsbachReiterVonderMosel2023}.
{For $\beta\ge 2$ the F\'ary--Milnor inequality was established
by Wacker~\cite{wacker2020,wacker2021}
under the assumption that $u$ has only
isolated self-intersection points.}
In the present paper
\cref{ass:GFM} is an explicit hypothesis, \emph{not} a claimed theorem.

Our main result is the following.

\begin{theorem}[Elastic BB knots]\label{thm:main-intro}
Let $\K$ be a one-component tame BB knot class and set
\[
  \beta:=\beta(\K)=\operatorname{braid}(\K).
\]
Suppose that \cref{ass:GFM} holds.  For every sequence
$\vartheta_j\downarrow0$ and
every choice of global minimizers $\gamma_{\vartheta_j}$ of
\eqref{eq:intro-total} in $\K$, every subsequential elastic limit is, up to
a Euclidean isometry and reparametrization,
\begin{equation}\label{eq:covered-circle-intro}
  u_\beta(s)
  =c+\frac{1}{2\pi\beta}
  \bigl(e_1\cos(2\pi\beta s)+e_2\sin(2\pi\beta s)\bigr),
\end{equation}
where $e_1,e_2\in\R^3$ are orthonormal.  In particular, the image is a
round circle and the parametrization covers it exactly $\beta$ times.
Moreover, along every convergent subsequence one has strong convergence in
$H^2(\T,\R^3)$.
\end{theorem}

The proof has two quantitative ingredients.  First, fibrewise contraction
of a fixed minimal closed braid gives smooth representatives $v_\rho\in\K$
with
\begin{equation}\label{eq:intro-recovery}
  \Fbar(v_\rho)-(2\pi\beta)^2=O(\rho^2),
  \qquad
  \Rop(v_\rho)=O(\rho^{-1}),
\end{equation}
where $\Fbar(\gamma):=\Length(\gamma)\Eb(\gamma)$ is the normalized,
scale-invariant bending energy.  The qualitative convergence of shrinking
braids to a multiply covered circle is discussed in
\cite[Section~2.3]{DiaoErnstReiter2021}; the rates in
\eqref{eq:intro-recovery} require a separate proof and are established in
\cref{sec:braid-recovery} below.

Second, if a unit loop $\gamma$ has thickness $\Delta$,
in short $\Thi(\gamma)=\Delta$, then for every
fixed smooth variation $h$ there is a constant $c_{h}>0$ such that
\begin{equation}\label{eq:intro-stability}
  \Thi(\gamma+r\varepsilon h)\ge c_h\Delta
  \quad\text{for }|r|\le1,
  \qquad |\varepsilon|\le c_h\Delta.
\end{equation}
This is a finite geometric estimate; no derivative of thickness or
ropelength occurs.  The 
energy decay of the minimizers $\gamma_\vartheta$ implies
$\Delta[\gamma_\vartheta]\gtrsim\vartheta^{1/3}$.  Choosing
$\varepsilon_\vartheta=\vartheta^{1/2}$
makes $\gamma_\vartheta\pm\varepsilon_\vartheta h$ admissible while the
ropelength contribution, divided by $\varepsilon_\vartheta$, tends to zero.
The two signs yield %
that the elastic limit {satisfies the Euler--Lagrange equation
for $\Fbar$}.
Combined with constant curvature, free stationarity gives a linear ODE for
the curvature vector and hence \eqref{eq:covered-circle-intro}.

This argument provides two useful shortcuts.  It avoids both a contact-force
or contact-graph analysis and any first variation of the nonsmooth
ropelength functional.  It also avoids invoking the classification or
stability theory of closed elasticae: after saturation of
\eqref{eq:bending-lower-intro}, the Euler--Lagrange equation becomes linear.

The proof is organized as follows.  We collect notation and standard facts
in \cref{sec:preliminaries}, {report on the} 
existence of regularized minimizers in
\cref{sec:existence}, and prove the quantitative closed-braid recovery in
\cref{sec:braid-recovery}.  The energy and thickness scales and strong
$H^2$-convergence are derived in \cref{sec:saturation}.  The derivative-free
finite-competitor argument occupies \cref{sec:finite-competitors}.  The
classification and the proof of \cref{thm:main-intro} are completed in
\cref{sec:classification}.

\subsection*{Usage of LLM}

Part of the work leading to this article is assisted by ChatGPT. All mathematical validation, final proof decisions, and final wording remain the sole responsibility of the human authors.

\subsection*{Acknowledgement}

The second author’s work is partially funded by the Excellence Initiative of the German federal and state governments.

\section{Geometric and variational preliminaries}\label{sec:preliminaries}

\subsection{Regular curves and bending energy}

Throughout, $\T=\R/\mathbb Z$ is equipped with the intrinsic distance
\[
  |s-t|_{\T}:=\min_{k\in\mathbb Z}|s-t+k|.
\]
For a regular curve $\gamma\in H^2(\T,\R^3)$, let
\begin{align}
  \Length(\gamma)&:=\int_{\T}|\gamma'(t)|\,dt,\label{eq:length}\\
  \Eb(\gamma)&:=\int_{\T}
  \frac{|\gamma'(t)\times\gamma''(t)|^2}{|\gamma'(t)|^5}\,dt,
  \label{eq:bending-general}\\
  \Fbar(\gamma)&:=\Length(\gamma)\Eb(\gamma).
  \label{eq:normalized-bending}
\end{align}
The functional $\Fbar$ is invariant under regular reparametrizations,
translations, rotations, and dilations.  If $\gamma$ is parametrized by
arclength and has length one, then
\begin{equation}\label{eq:unit-bending}
  \Fbar(\gamma)=\Eb(\gamma)=\int_{\T}|\gamma''(s)|^2\,ds.
\end{equation}

The total curvature is
\begin{equation}\label{eq:total-curvature}
  \TC(\gamma):=\int_\gamma\kappa\,ds
  =\int_{\T}\frac{|\gamma'(t)\times\gamma''(t)|}{|\gamma'(t)|^2}\,dt.
\end{equation}
Cauchy--Schwarz gives the scale-invariant estimate
\begin{equation}\label{eq:CS-total-curvature}
  \TC(\gamma)^2\le \Fbar(\gamma).
\end{equation}
For a unit-speed unit loop, equality holds in
\eqref{eq:CS-total-curvature} if and only if $|\gamma''|$ is constant
almost everywhere.

Let
\[
  H^2_{\mathrm{reg}}\equiv{}
  H^2_{\mathrm{reg}}(\T,\R^3)
  :=\{\gamma\in H^2(\T,\R^3):\min_{\T}|\gamma'|>0\}.
\]
This is open in $H^2$, because $H^2(\T)\hookrightarrow C^1(\T)$.

\begin{lemma}[Smoothness and first variation of $\Fbar$]
\label{lem:Fbar-smooth}
The functional $\Fbar$ is smooth on
$H^2_{\mathrm{reg}}(\T,\R^3)$.  If $u$ is a unit-speed unit loop, and
$T:=u'$, $K:=u''$, then for every $h\in H^2(\T,\R^3)$,
\begin{equation}\label{eq:Fbar-first-general}
  D\Fbar(u)[h]
  =\Eb(u)\int_{\T}T\cdot h'\,ds
   +2\int_{\T}K\cdot h''\,ds
   -3\int_{\T}|K|^2T\cdot h'\,ds.
\end{equation}
If in addition $|K|\equiv\lambda$, then
\begin{equation}\label{eq:Fbar-first-constant}
  D\Fbar(u)[h]
  =2\int_{\T}\bigl(K\cdot h''-\lambda^2T\cdot h'\bigr)\,ds.
\end{equation}
\end{lemma}

\begin{proof}
For $\gamma\in H^2_{\mathrm{reg}}$, set
$v:=|\gamma'|$, $T_\gamma:=\gamma'/v$, and
$P_{T_\gamma}^\perp:=I-T_\gamma\otimes T_\gamma$.  Formula
\eqref{eq:bending-general} can be written as
\begin{equation}\label{eq:bending-projection}
  \Eb(\gamma)
  =\int_{\T}\frac{|P_{T_\gamma}^\perp\gamma''|^2}{v^3}\,dt.
\end{equation}
The maps $\gamma\mapsto v^{-1}$ and
$\gamma\mapsto P_{T_\gamma}^\perp$ are smooth on the open set where
$\min|\gamma'|>0$.  Since $H^1(\T)$ is a Banach algebra and
$H^2(\T)\hookrightarrow C^1(\T)$, the right-hand sides of
\eqref{eq:length} and \eqref{eq:bending-projection} define smooth maps
on $H^2_{\mathrm{reg}}$.  In particular, the second derivative of
$\Fbar$ is locally bounded; this observation will justify the quadratic
Taylor estimate used in \cref{prop:braid-recovery}.

At a unit-speed curve {$u$} one has
\begin{equation}\label{eq:length-variation}
  D\Length(u)[h]=\int_{\T}T\cdot h'\,ds.
\end{equation}
Moreover, $K\cdot T=0$, and differentiating
\eqref{eq:bending-projection} gives
\begin{equation}\label{eq:bending-variation}
  D\Eb(u)[h]
  =2\int_{\T}K\cdot h''\,ds
   -3\int_{\T}|K|^2T\cdot h'\,ds.
\end{equation}
Indeed, the variation of $P_{T_u}^\perp K$ contributes only a tangential
vector and hence has zero scalar product with $K$.  Combining
\eqref{eq:length-variation} and \eqref{eq:bending-variation} with the
product rule proves \eqref{eq:Fbar-first-general}.  If $|K|=\lambda$ and
$\Length(u)=1$, then $\Eb(u)=\lambda^2$, and
\eqref{eq:Fbar-first-constant} follows.
\end{proof}

\subsection{Thickness and ropelength}

For a regular $C^1$-curve $\gamma$, write
$T_\gamma(s):=\gamma'(s)/|\gamma'(s)|$.  For $s\ne t$, define the
tangent-point radius
\begin{equation}\label{eq:tangent-point-radius}
  r_{\mathrm{tp}}[\gamma](s,t)
  :=\frac{|\gamma(t)-\gamma(s)|^2}
  {2|P_{T_\gamma(s)}^\perp(\gamma(t)-\gamma(s))|},
\end{equation}
with value zero if $\gamma(t)=\gamma(s)$, and with value $+\infty$ if
the chord is nonzero but the denominator vanishes.  The thickness is
\begin{equation}\label{eq:thickness-global-radius}
  \Thi(\gamma):=\inf_{s\ne t}r_{\mathrm{tp}}[\gamma](s,t).
\end{equation}
For embedded $C^{1,1}$-curves this agrees with the reach and with the
normal injectivity radius.  It also admits the characterization
{(see\cite{LitherlandSimonDurumericRawdon1999})}
\begin{equation}\label{eq:thickness-dcsd}
  \Thi(\gamma)
  =\min\left\{\MinRad(\gamma),\frac12\dcsd(\gamma)\right\},
\end{equation}
where $\MinRad$ is the essential infimum of the radius of curvature and
$\dcsd$ is the infimum of the lengths of chords perpendicular to the
tangents at both endpoints.  These facts, including the 
$C^{1,1}$-regularity
implied by positive thickness, are proved in
\cite{zbMATH01743161,CantarellaKusnerSullivan2002,zbMATH01953583}.

Ropelength is
\begin{equation}\label{eq:ropelength}
  \Rop(\gamma):=\frac{\Length(\gamma)}{\Thi(\gamma)},
\end{equation}
with the convention $\Rop(\gamma)=+\infty$ if
$\Thi(\gamma)=0$.  It is invariant under reparametrization and dilation.

We shall repeatedly use the following chord--arc estimate.

\begin{lemma}[Thick chord--arc estimate]\label{lem:chord-arc}
Let $\gamma$ be a unit-speed embedded $C^{1,1}$-curve of thickness
$\Delta>0$, and let $w=|s-t|_{\T}$ be the length of its shorter
arc connecting~$\gamma(s)$ with~$\gamma(t)$. 
If $|\gamma(s)-\gamma(t)|<2\Delta$, then
\begin{equation}\label{eq:chord-arc}
  w\le 2\Delta\arcsin
  \frac{|\gamma(s)-\gamma(t)|}{2\Delta}.
\end{equation}
Consequently, for every $a\in(0,\pi]$,
\begin{equation}\label{eq:chord-arc-converse}
  w\ge a\Delta
  \quad\Longrightarrow\quad
  |\gamma(s)-\gamma(t)|\ge2\Delta\sin(a/2).
\end{equation}
\end{lemma}

\begin{proof}
The first assertion is \cite[Lemma~5]{CantarellaKusnerSullivan2002}.
If the conclusion of \eqref{eq:chord-arc-converse} failed, then
\eqref{eq:chord-arc}, together with the strict monotonicity of
$\arcsin$ on $[0,1]$, would give $w<a\Delta$, a contradiction.
\end{proof}

\subsection{Knot classes, unit loops, and BB knots}

Let $\K$ be a one-component tame knot class.  We work with the based
unit loops
\begin{equation}\label{eq:unit-loop-class}
  \Cclass(\K):=
  \left\{\gamma\in H^2(\T,\R^3):
  \begin{array}{l}
  \gamma(0)=0,\quad |\gamma'|\equiv1,\\[-1mm]
  \gamma\text{ is embedded and represents }\K
  \end{array}\right\}.
\end{equation}
Every curve in this set has length one.  The base-point condition only
removes translations and has no geometric significance.

The bridge index is denoted by $\beta(\K)$ and the braid index by
$\operatorname{braid}(\K)$.  We call $\K$ a one-component BB knot class
if
\[
  \beta(\K)=\operatorname{braid}(\K).
\]
For a one-component knot, reversal of orientation does not change the
braid index; cf.~\cite[p.~2150075-3, last sentence]{DiaoErnstReiter2021}.
For links, the unoriented convention in
\cite[p.~2150075-3, last paragraph]{DiaoErnstReiter2021} involves a minimum over component orientations;
we do not treat that case here.

\begin{lemma}[Ordinary bridge lower bound]\label{lem:ordinary-bridge}
If $\gamma\in\Cclass(\K)$, then
\begin{equation}\label{eq:ordinary-bridge}
  \TC(\gamma)\ge2\pi\beta(\K),
  \qquad
  \Eb(\gamma)\ge\bigl(2\pi\beta(\K)\bigr)^2.
\end{equation}
\end{lemma}

\begin{proof}
For a smooth embedded representative, the first inequality is Milnor's
bridge-total-curvature inequality \cite{Milnor1950}.
Exploiting the fact that any $H^{2}$ curve is of finite total curvature,
the first inequality follows from~\cite[Cor.~4.4]{zbMATH05282830}.
For completeness, we give an independent proof. Let $\gamma\in\Cclass(\K)$ merely be of class $H^2$.  Periodic smoothing
and an arbitrarily small regular correction produce smooth regular curves
$\gamma_j\to\gamma$ strongly in $H^2$ and hence in $C^1$.  Since $\gamma$
is a $C^1$-embedding, all sufficiently large $\gamma_j$ are embedded and
ambient isotopic to $\gamma$ due to the isotopy stability theorem; %
see~\cite[Lemma~3$\cdot$2]{DiaoErnstJanseVanRensburg1998} as well as \cite[Theorem~A]{BlattGilsbachReiterVonderMosel2025} and references therein.
Furthermore, on a $C^1$-neighborhood on
which the speed is bounded away from zero, the integrand in
\eqref{eq:total-curvature} is continuous under strong $H^2$-convergence
in $L^1$.  Hence $\TC(\gamma_j)\to\TC(\gamma)$, and the first inequality
passes to the limit.  The second follows from
\eqref{eq:CS-total-curvature} and \eqref{eq:unit-bending}.
\end{proof}

\section{Regularized minimizers and elastic limits}\label{sec:existence}

For $\vartheta>0$, define on $\Cclass(\K)$
\begin{equation}\label{eq:Etheta}
  E_\vartheta(\gamma)
  :=\Eb(\gamma)+\vartheta\Rop(\gamma).
\end{equation}

For every tame knot class $\K$ and every $\vartheta>0$, the functional
$E_\vartheta$ has a global minimizer
$\gamma_\vartheta\in\Cclass(\K)$ by
\cite[Theorem~2.1]{GerlachReiterVonderMosel2017}.

\begin{definition}[Elastic knot]\label{def:elastic-knot}
Let $\vartheta_j\downarrow0$, and for each $j$ let
$\gamma_{\vartheta_j}$ be a global minimizer of $E_{\vartheta_j}$.
Every weak $H^2$- and strong $C^1$-subsequential limit is called an
\emph{elastic knot} for $\K$.
\end{definition}

Such subsequential limits exist, see~\cite[Theorem~2.2]{GerlachReiterVonderMosel2017}. Indeed, comparison with any fixed smooth unit
representative gives a uniform bending-energy bound, and therefore an
$H^2$-bound.  They need not be embedded,
in fact, if $\K$ is nontrivial, they exhibit self-intersections, cf.~\cite[Prop.~3.1]{GerlachReiterVonderMosel2017}.  This loss of embeddedness is
the reason that \cref{ass:GFM} is needed.

\section{Quantitative recovery by thin closed braids}
\label{sec:braid-recovery}

We now prove the topology-dependent part of the argument.  The conclusion
holds for every one-component knot class with braid index $n$, independently
of its bridge index.  The BB condition will enter only when the limiting
bending energy $(2\pi n)^2$ is compared with the bridge lower bound.

For $n\in\mathbb N$, define the unit-length $n$-covered circle
\begin{equation}\label{eq:un-covered-circle}
  u_n(t):=R_n\bigl(e_r(2\pi nt)-e_r(0)\bigr),
  \qquad
  R_n:=\frac{1}{2\pi n},
  \qquad
  e_r(\theta):=(\cos\theta,\sin\theta,0).
\end{equation}
Then $|u_n'|\equiv1$, $|u_n''|\equiv2\pi n$, and
\begin{equation}\label{eq:circle-energy}
  \Fbar(u_n)=\Eb(u_{n})=(2\pi n)^2.
\end{equation}

\begin{lemma}[Stationarity of a multiply covered circle]
\label{lem:circle-stationary}
For every $n\in\mathbb N$ and every $h\in H^2(\T,\R^3)$,
\begin{equation}\label{eq:circle-stationary}
  D\Fbar(u_n)[h]=0.
\end{equation}
\end{lemma}

\begin{proof}
Set $\lambda_n:=2\pi n$, $T:=u_n'$, and $K:=u_n''$.  Directly from
\eqref{eq:un-covered-circle},
\[
  K''+\lambda_n^2K=0,
  \qquad T'=K.
\]
By \eqref{eq:Fbar-first-constant} and periodic integration by parts,
\[
  \frac12D\Fbar(u_n)[h]
  =\int_\T(K''+\lambda_n^2K)\cdot h\,ds=0.
\]
\end{proof}

For torus knots, the corresponding bending and ropelength estimates are
\cite[Lemma~5.1]{GerlachReiterVonderMosel2017} and
\cite[Proposition~5.2]{GerlachReiterVonderMosel2017}, respectively.  The
next result extends both to arbitrary closed braids.

\begin{proposition}[Thin closed-braid recovery]
\label{prop:braid-recovery}
Let $\K$ be a one-component knot class of braid index $n$.  There exist
$\rho_0,c,C>0$ and smooth curves
$v_\rho\in\Cclass(\K)$, $0<\rho\le\rho_0$, such that
\begin{align}
  v_\rho&\longrightarrow u_n
  &&\text{smoothly as }\rho\downarrow0,
  \label{eq:braid-smooth-convergence}\\
  |\Fbar(v_\rho)-(2\pi n)^2|&\le C\rho^2,
  \label{eq:braid-bending-rate}\\
  \Thi(v_\rho)&\ge c\rho,
  \qquad
  \Rop(v_\rho)\le C\rho^{-1}.
  \label{eq:braid-rope-rate}
\end{align}
The constants may depend on the chosen geometric braid representative of
$\K$.
\end{proposition}

\begin{proof}
We separate the construction, bending estimate, and thickness estimate.
If $n=1$, take $v_\rho:=u_1$ for every $\rho$; all conclusions are
immediate after decreasing $\rho_0$ if necessary.  We may therefore
assume $n\ge2$.

\smallskip
\noindent\emph{Step 1: fibre contraction.}
Let
\[
  c(\theta):=R_ne_r(\theta),
  \qquad e_z:=(0,0,1),
\]
and choose a tube radius $r_0<R_n$.  The map
\begin{equation}\label{eq:tube-map}
  \Psi\colon(\R/2\pi\mathbb Z)\times D_{r_0}\longrightarrow\R^3,
  \qquad
  \Psi(\theta,x,y):=(R_n+x)e_r(\theta)+ye_z,
\end{equation}
is a smooth embedding.
Here $D_{r}$ denotes a two-dimensional round disk of radius $r>0$.

By the definition of braid index, $\K$ has a smooth geometric closed
$n$-braid $B$ in this solid torus.  The angular projection
$B\to\R/2\pi\mathbb Z$ is an $n$-sheeted covering.  Since $B$ is
connected, its monodromy is one $n$-cycle.  Following the knot once and
lifting the angular coordinate therefore gives a smooth function
$z\colon\R\to D_{r_0}$ of period $2\pi n$ such that
\begin{equation}\label{eq:braid-lift-separation}
  z(\theta+2\pi j)\ne z(\theta)
  \quad\text{for all }\theta\in\R,
  \quad j=1,\ldots,n-1.
\end{equation}
Put $w(t):=z(2\pi nt)$ and write $w=(x,y)$.  Then $w$ is one-periodic,
and compactness together with \eqref{eq:braid-lift-separation} gives
\begin{equation}\label{eq:d-star}
  d_*:=\min_{t\in\T}\min_{1\le j<n}
  |w(t+j/n)-w(t)|>0.
\end{equation}

After an initial harmless transverse rescaling of $B$, define
\begin{equation}\label{eq:Gamma-rho}
  \begin{aligned}
    \widetilde\Gamma_\rho(t)
    &:=(R_n+\rho x(t))e_r(2\pi nt)+\rho y(t)e_z,\\
    \Gamma_\rho(t)
    &:=\widetilde\Gamma_\rho(t)-\widetilde\Gamma_\rho(0),
    \qquad 0<\rho\le1.
  \end{aligned}
\end{equation}
For sufficiently small $\rho$ the radial coordinate is positive.  If
$\Gamma_\rho(s)=\Gamma_\rho(t)$, then
$\widetilde\Gamma_\rho(s)=\widetilde\Gamma_\rho(t)$.  Their angular
coordinates therefore agree, so $t=s+j/n$ modulo one.  The remaining
two coordinates and
\eqref{eq:d-star} imply $j=0$ and then $s=t$.  Hence every
$\Gamma_\rho$ is embedded.  The family
$\rho\mapsto\Gamma_\rho$, restricted to positive $\rho$, is an isotopy
through translates of closed $n$-braids; the isotopy extension theorem
shows that all $\Gamma_\rho$ represent $\K$.  Moreover,
$\Gamma_\rho(0)=0$.  At $\rho=0$ the family converges smoothly to $u_n$,
which is regular but nonembedded for $n>1$.

\smallskip
\noindent\emph{Step 2: quadratic bending excess.}
By \eqref{eq:Gamma-rho},
\begin{equation}\label{eq:Gamma-affine}
  \Gamma_\rho=u_n+\rho h
\end{equation}
for one fixed smooth periodic $h$.  The curves remain regular for
$|\rho|$ small.  By \cref{lem:Fbar-smooth,lem:circle-stationary}, Taylor's
formula gives
\begin{equation}\label{eq:Gamma-energy-expansion}
  \Fbar(\Gamma_\rho)
  =\Fbar(u_n)+\rho D\Fbar(u_n)[h]+O(\rho^2)
  =(2\pi n)^2+O(\rho^2).
\end{equation}

\smallskip
\noindent\emph{Step 3: linear thickness bound.}
Smooth convergence to the regular immersion $u_n$ implies uniform bounds
\begin{equation}\label{eq:Gamma-C2-bounds}
  |\Gamma_\rho'|\ge v_*>0,
  \qquad
  |\Gamma_\rho''|\le M
\end{equation}
for $0<\rho\le\rho_0$.  Thus the curvature is bounded uniformly and
\begin{equation}\label{eq:Gamma-minrad}
  \MinRad(\Gamma_\rho)\ge r_*>0.
\end{equation}

The bounds \eqref{eq:Gamma-C2-bounds} also provide a number
$\delta_0>0$, independent of $\rho$, such that no pair with
$0<|s-t|_\T<\delta_0$ is doubly critical.  Indeed, in a local parameter
chart, with $\tau=t-s$,
\begin{equation}\label{eq:local-chord-tangent}
  \bigl(\Gamma_\rho(t)-\Gamma_\rho(s)\bigr)
  \cdot\Gamma_\rho'(s)
  =\tau|\Gamma_\rho'(s)|^2+O(\tau^2),
\end{equation}
where the remainder is uniform in $s$ and $\rho$.  The right-hand side
has the sign of $\tau$ if $0<|\tau|<\delta_0$.

It remains to bound chords with $|s-t|_\T\ge\delta_0$.  Let
$\alpha\in[-\pi,\pi]$ be determined by
\[
  \alpha\equiv2\pi n(t-s)\pmod{2\pi}.
\]
Use the orthonormal frame
\[
  \bigl(e_r(2\pi ns),e_\theta(2\pi ns),e_z\bigr),
  \qquad
  e_\theta(\theta):=(-\sin\theta,\cos\theta,0).
\]
In this frame one has
\begin{align}
  \Gamma_\rho(t)-\Gamma_\rho(s)
  ={}&\bigl[(R_n+\rho x(t))\cos\alpha
       -(R_n+\rho x(s))\bigr]e_r(2\pi ns)
  \notag\\
  &+(R_n+\rho x(t))\sin\alpha\,e_\theta(2\pi ns)
  \notag\\
  &+\rho(y(t)-y(s))e_z.
  \label{eq:exact-braid-chord}
\end{align}
If $|\sin\alpha|\ge\rho$, the $e_\theta$ component yields
\begin{equation}\label{eq:chord-sine-case}
  |\Gamma_\rho(t)-\Gamma_\rho(s)|\ge c\rho.
\end{equation}
If $|\sin\alpha|<\rho$ and $\alpha$ is near $\pi$ or $-\pi$, the
radial component in \eqref{eq:exact-braid-chord} is bounded away from
zero, so \eqref{eq:chord-sine-case} remains true.  The only remaining
case is $|\alpha|\le C\rho$.

In that case there is a unique $j\in\{0,\ldots,n-1\}$ such that
$t_0:=s+j/n$ has the same angular coordinate as $s$ and
$|t-t_0|\le C\rho$.  Since $|s-t|_\T\ge\delta_0$, one has $j\ne0$ for
small $\rho$.  At $t_0$, the normal-plane part of the chord equals
$\rho(w(t_0)-w(s))$ and has length at least $\rho d_*$.  Smoothness of
$w$ and \eqref{eq:exact-braid-chord} show that replacing $t_0$ by $t$
changes this normal part by at most
\begin{equation}\label{eq:normal-error}
  C(\alpha^2+\rho|\alpha|)\le C\rho^2.
\end{equation}
Consequently,
\begin{equation}\label{eq:nonlocal-chord-rho}
  |\Gamma_\rho(t)-\Gamma_\rho(s)|
  \ge\rho d_*-C\rho^2\ge\frac{d_*}{2}\rho
\end{equation}
for sufficiently small $\rho$.

Every doubly critical chord is nonlocal by
\eqref{eq:local-chord-tangent}, and hence
$\dcsd(\Gamma_\rho)\ge c\rho$ by
\eqref{eq:chord-sine-case}--\eqref{eq:nonlocal-chord-rho}.  Combining
this with \eqref{eq:Gamma-minrad} and
\eqref{eq:thickness-dcsd} proves
\begin{equation}\label{eq:Gamma-thickness}
  \Thi(\Gamma_\rho)\ge c\rho.
\end{equation}
Since $\Length(\Gamma_\rho)$ stays bounded,
$\Rop(\Gamma_\rho)\le C/\rho$.

Finally, dilate $\Gamma_\rho$ to length one and reparametrize it by
arclength, starting at parameter value zero.  Call the result $v_\rho$.
Normalized
bending energy and ropelength are invariant under these operations, and
the dilation factors stay bounded above and below.  Thus
\eqref{eq:braid-bending-rate}--\eqref{eq:braid-rope-rate} follow from
\eqref{eq:Gamma-energy-expansion} and \eqref{eq:Gamma-thickness}.
Smooth convergence in \eqref{eq:braid-smooth-convergence} is preserved by
the normalization.
\end{proof}

\begin{remark}\label{rem:recovery-not-simultaneous}
The braid representative used in \cref{prop:braid-recovery} need not
simultaneously realize bridge position.  This is immaterial: the recovery
upper bound uses a braid-minimizing representative, while the lower bound
uses the bridge index as a knot invariant.
\end{remark}

\section{{Energy decay of minimizers, and strong convergence in $H^2$}}
\label{sec:saturation}

{We now use the comparison curves contructed in the previous section
to prove energy decay rates for the minimizing knots $\gamma_\vartheta$.}

\begin{proposition}[Energy and thickness scales]\label{prop:scales}
Let $\K$ be a one-component BB knot class with bridge index $\beta$, set
$\lambda:=2\pi\beta$, and let
$\gamma_\vartheta$ minimize $E_\vartheta$ in $\Cclass(\K)$.  Then, for
all sufficiently small $\vartheta>0$,
\begin{align}
  0\le \Eb(\gamma_\vartheta)-\lambda^2
  &\le C\vartheta^{2/3},
  \label{eq:scale-bending}\\
  \Rop(\gamma_\vartheta)
  &\le C\vartheta^{-1/3},
  \label{eq:scale-rope}\\
  \Delta_\vartheta:=\Thi(\gamma_\vartheta)
  &\ge c\vartheta^{1/3}.
  \label{eq:scale-thickness}
\end{align}
\end{proposition}

\begin{proof}
This is the comparison argument of
\cite[Proposition~5.4]{GerlachReiterVonderMosel2017}, with the explicit
torus-knot family replaced by the general closed-braid recovery from
\cref{prop:braid-recovery}.  Indeed, since
$\operatorname{braid}(\K)=\beta$, apply
\cref{prop:braid-recovery} with $n=\beta$ and choose
$\rho=\vartheta^{1/3}$.  Because the comparison curves have length one,
$\Eb(v_\rho)=\Fbar(v_\rho)$, and minimality gives
\begin{equation}\label{eq:comparison-scale}
  E_\vartheta(\gamma_\vartheta)
  \le E_\vartheta(v_\rho)
  \le\lambda^2+C\bigl(\rho^2+\vartheta\rho^{-1}\bigr)
  \le\lambda^2+C\vartheta^{2/3}.
\end{equation}
Applying \cref{lem:ordinary-bridge} to the embedded curve
$\gamma_\vartheta$ gives $\Eb(\gamma_\vartheta)\ge\lambda^2$.
Dropping the nonnegative
ropelength term in \eqref{eq:comparison-scale} proves
\eqref{eq:scale-bending}.  Subtracting the same lower bound from
\eqref{eq:comparison-scale} gives
\[
  \vartheta\Rop(\gamma_\vartheta)
  \le C\vartheta^{2/3},
\]
which is \eqref{eq:scale-rope}.  Since the curves have length one,
$\Delta_\vartheta=1/\Rop(\gamma_\vartheta)$, proving
\eqref{eq:scale-thickness}.
\end{proof}

\begin{proposition}[Strong convergence and constant curvature]
\label{prop:strong-convergence}
Under the assumptions of \cref{prop:scales}, suppose in addition that
\cref{ass:GFM} holds, and let
$\vartheta_j\downarrow0$.  After passing to a subsequence, there is a
unit-speed $u\in H^2(\T,\R^3)$ such that
\begin{equation}\label{eq:strong-H2}
  \gamma_{\vartheta_j}\longrightarrow u
  \quad\text{strongly in }H^2(\T,\R^3),
\end{equation}
and
\begin{equation}\label{eq:constant-curvature-limit}
  \TC(u)=\lambda,
  \qquad
  |u''|=\lambda\quad\text{a.e. on }\T,
  \qquad
  \Eb(u)=\lambda^2.
\end{equation}
\end{proposition}

\begin{proof}
The compactness argument is the one used in
\cite[Theorem~2.2]{GerlachReiterVonderMosel2017}, while the equality
argument below generalizes
\cite[Proposition~3.2]{GerlachReiterVonderMosel2017} from the lower bound
$4\pi$ to $2\pi\beta$.  We additionally record the resulting strong
$H^2$-convergence; cf.~\cite[Cor.~3.5]{GerlachReiterVonderMosel2017}.

The uniform bending bound and the based unit-speed normalization give a
uniform $H^2$-bound.  Thus, after taking a subsequence,
\begin{equation}\label{eq:weak-strong-prelimit}
  \gamma_{\vartheta_j}\weakto u\quad\text{in }H^2,
  \qquad
  \gamma_{\vartheta_j}\to u\quad\text{in }C^1.
\end{equation}
In particular, $|u'|\equiv1$, and $u$ belongs to the $C^1$-closure of
$\Cclass(\K)$.  By \cref{ass:GFM}, Cauchy--Schwarz, weak lower
semicontinuity, and \eqref{eq:scale-bending},
\begin{equation}\label{eq:equality-chain}
  \lambda^2
  \le\TC(u)^2
  \le\Eb(u)
  \le\liminf_{j\to\infty}\Eb(\gamma_{\vartheta_j})
  \le\limsup_{j\to\infty}\Eb(\gamma_{\vartheta_j})
  \le\lambda^2.
\end{equation}
Every inequality is therefore an equality.  Equality in Cauchy--Schwarz
gives $|u''|=\lambda$ almost everywhere.  Moreover,
\[
  \|\gamma_{\vartheta_j}''\|_{L^2}^2
  \longrightarrow\lambda^2=\|u''\|_{L^2}^2.
\]
Together with weak convergence of the second derivatives, the Hilbert
space norm criterion yields strong $L^2$-convergence of the second
derivatives.  This and \eqref{eq:weak-strong-prelimit} prove
\eqref{eq:strong-H2}.
\end{proof}

\section{Finite competitors and derivative-free removal of ropelength}
\label{sec:finite-competitors}

The next lemma is the geometric core of the proof.  It is deliberately
formulated as a finite estimate.  In particular, it neither asserts nor uses
local Lipschitz continuity of ropelength as a functional on $H^2$.

\begin{lemma}[%
Thickness stability]
\label{lem:multiplicative-thickness}
Let $h\in C^2(\T,\R^3)$ be fixed.  There are constants
$a_h,c_h,C_h{'}>0$ with the following property.  If $\gamma$ is a
unit-speed embedded $C^{1,1}$-curve of length one and thickness
$\Delta>0$, and if
\begin{equation}\label{eq:epsilon-thickness-condition}
  |\varepsilon|\le a_h\Delta,
\end{equation}
then, for every $r\in[-1,1]$, the curve
\begin{equation}\label{eq:eta-r}
  \eta_r:=\gamma+r\varepsilon h
\end{equation}
is a regular embedded $C^{1,1}$-curve ambient isotopic to $\gamma$, and
\begin{equation}\label{eq:multiplicative-thickness}
  \Thi(\eta_r)\ge c_h\Delta,
  \qquad
  \Rop(\eta_r)\le\frac{C_h'}{\Delta}.
\end{equation}
The constants are independent of $\gamma$ and $\Delta$.
\end{lemma}

\begin{proof}
Positive thickness and \eqref{eq:thickness-dcsd} imply
(cf.~\cite[Lemma~2]{zbMATH01743161})
\begin{equation}\label{eq:gamma-curvature-thickness}
  \|\gamma''\|_{L^\infty}\le\Delta^{-1}.
\end{equation}
Fenchel's theorem and \eqref{eq:gamma-curvature-thickness} also give the
universal bound $\Delta\le(2\pi)^{-1}$, because $\gamma$ has length one.
In particular we have that
\begin{equation}\label{eq:delta-bound}
\Delta\le1\le\frac1\Delta.
\end{equation}
Letting
\[ 0<a_{h}\le\min\Big(1, \frac{\pi}{\|h'\|_{L^{\infty}}+1}\Big), \]%
\eqref{eq:epsilon-thickness-condition} implies,
uniformly for $|r|\le1$,
\begin{equation}\label{eq:eta-speed-curvature}
  \frac12\le|\eta_r'|\le\frac32,
  \qquad
  \|\kappa_{\eta_r}\|_{L^\infty}
  \le\frac{C_h}{\Delta}.
\end{equation}
Indeed, using~\eqref{eq:delta-bound},
{\[
|\eta_r''|
  \le
  (\Delta^{-1}+|\varepsilon|\|h''\|_{L^\infty})
  \le (\Delta^{-1}+a_{h}\Delta\|h''\|_{L^\infty})
  \le \Delta^{-1}(1+\|h''\|_{L^\infty})
\]}%
and
\[
  \kappa_{\eta_r}
  \le\frac{|\eta_r''|}{|\eta_r'|^2}
  \le {4\Delta^{-1}(1+\|h''\|_{L^\infty}) =: \frac{C_h}{\Delta}}.
\]

We next distinguish local and nonlocal parameter pairs.  Let
$w:=|s-t|_\T$.  By %
{the bound on $\eta_{r}''$}, the derivative
$\eta_r'$ is Lipschitz with constant at most $C_h/\Delta$.  In a local
chart, orient the shorter interval from $s$ to $t$ and %
let $\sigma=\sign(t-s)$.  Then
\begin{align}
  \sigma\bigl(\eta_r(t)-\eta_r(s)\bigr)\cdot\eta_r'(s)
  &=\sigma\int_s^t\eta_r'(q)\cdot\eta_r'(s)\,dq
  \notag\\
  &\ge\sigma\int_s^t
  \left(|\eta_r'(s)|^2
  -|\eta_r'(q)-\eta_r'(s)|\,|\eta_r'(s)|\right)dq
  \label{eq:local-dot-lower}\\
  &\ge|w|\bigg(\tfrac14-\tfrac32|w|\frac{C_{h}}\Delta\bigg).\notag
\end{align}
Choose $\alpha%
{{}=\min(\pi,1/(6C_{h}))>0}$.  If
$0<w<\alpha\Delta$, then the right-hand side of
\eqref{eq:local-dot-lower} is positive by \eqref{eq:eta-speed-curvature}.
Thus a local chord is nonzero and cannot be perpendicular to
the tangent at one of its endpoints.  In particular, no such pair is
doubly critical.

If $w\ge \alpha\Delta$, the chord--arc estimate
\eqref{eq:chord-arc-converse} for the original curve gives
\begin{equation}\label{eq:baseline-nonlocal-chord}
  |\gamma(t)-\gamma(s)|\ge2\Delta\sin(\alpha/2).
\end{equation}
{Additionally claiming that $a_{h}\le\frac{\sin(\alpha/2)}{\|h'\|_{L^{\infty}}+1}$}
we obtain
\begin{equation}\label{eq:perturbed-nonlocal-chord}
  |\eta_r(t)-\eta_r(s)|\ge {\Delta\sin(\alpha/2)%
  }%
  \qquad\text{whenever }w\ge \alpha\Delta.
\end{equation}

The local sign estimate and \eqref{eq:perturbed-nonlocal-chord} show first
that $\eta_r$ is injective.  They also show that every doubly critical
chord of $\eta_r$ has length at least $%
{\Delta\sin(\alpha/2)}$.  Combining this with
the curvature estimate in \eqref{eq:eta-speed-curvature} and the
thickness formula \eqref{eq:thickness-dcsd} proves the first estimate in
\eqref{eq:multiplicative-thickness}.  The length of $\eta_r$ is bounded
above by $3/2$, cf.~\eqref{eq:eta-speed-curvature}, %
and the
ropelength estimate in~\eqref{eq:multiplicative-thickness} follows.

All estimates are uniform in $r\in[-1,1]$, and
$r\mapsto\eta_r$ is continuous in $C^1$.  By the $C^1$-stability of
isotopy classes \cite[Lemma~3$\cdot$2]{DiaoErnstJanseVanRensburg1998}, the
ambient isotopy class of $\eta_r$ is locally constant in $r$.  Since $[-1,1]$ is
connected and $\eta_0=\gamma$, every $\eta_r$ is {ambient} isotopic to $\gamma$.
\end{proof}

\begin{remark}[Relation to earlier stability results]
Continuity of global curvature in the $C^{1,1}$-topology and isotopy
stability under a common global-curvature bound are proved in
\cite[Theorem~5 and Proposition~4]{SchurichtVonderMosel2003}; see also
\cite[Lemmas~2 {and 3}]{SchurichtVonderMosel2003Euler}.  The point of
\cref{lem:multiplicative-thickness} is the uniform, scale-invariant estimate:
its constants are independent of the reference curve and of its thickness.
\end{remark}

\begin{remark}\label{rem:fixed-h}
The constants in \cref{lem:multiplicative-thickness} depend on the fixed
$C^2$-norm of $h$.  No corresponding statement is claimed uniformly on
bounded sets of $H^2$-variations.  Such a claim would be false without
additional control, since high-frequency perturbations can create large
curvature.
\end{remark}

We now combine \cref{lem:multiplicative-thickness} with the two scales in
\cref{prop:scales}.

\begin{proposition}[Free stationarity of an elastic limit]
\label{prop:free-stationarity}
Under the assumptions of {\cref{prop:strong-convergence}}, let
$\gamma_{\vartheta_j}\to u$ strongly in $H^2$ be a subsequence as in
\cref{prop:strong-convergence}.  Then
\begin{equation}\label{eq:free-stationarity}
  D\Fbar(u)[h]=0
  \qquad\text{for every }h\in C^\infty(\T,\R^3).
\end{equation}
\end{proposition}

\begin{proof}
Put $\varepsilon_\vartheta:=\vartheta^{1/2}$.  By
\eqref{eq:scale-thickness},
\begin{align}
  \frac{\varepsilon_\vartheta}{\Delta_\vartheta}
  &\le C\vartheta^{1/6}\longrightarrow0,
  \label{eq:epsilon-over-thickness}\\
  \frac{\vartheta}
  {\varepsilon_\vartheta\Delta_\vartheta} {{}=\frac{\varepsilon_\vartheta}{\Delta_\vartheta}}
  &\le C\vartheta^{1/6}\longrightarrow0.
  \label{eq:penalty-over-epsilon}
\end{align}
Fix $h\in C^\infty(\T,\R^3)$.  For all sufficiently small
$\vartheta$, \cref{lem:multiplicative-thickness} applies to
\begin{equation}\label{eq:plus-minus-raw}
  \eta_\vartheta^\pm
  :=\gamma_\vartheta\pm\varepsilon_\vartheta h.
\end{equation}
Thus $\eta_\vartheta^\pm$ is ambient isotopic to $\gamma_\vartheta$ and
by virtue of~\eqref{eq:multiplicative-thickness}
\begin{equation}\label{eq:rope-new-competitor}
  \Rop(\eta_\vartheta^\pm)
  \le\frac{C_h}{\Delta_\vartheta}.
\end{equation}

The curves in \eqref{eq:plus-minus-raw} need not be based, unit-speed, or
of length one.  Translate each curve, dilate it by the reciprocal of its
length, and reparametrize it by arclength.  The resulting curve{s $\widehat\eta_\vartheta^\pm$} belong %
to $\Cclass(\K)$.  Reparametrization, dilation, and translation 
invariance give
\begin{equation}\label{eq:normalization-identities}
  \Eb(\widehat\eta_\vartheta^\pm)
  {{}=\Fbar(\widehat\eta_\vartheta^\pm)}
  =\Fbar(\eta_\vartheta^\pm),
  \qquad
  \Rop(\widehat\eta_\vartheta^\pm)
  =\Rop(\eta_\vartheta^\pm).
\end{equation}
Since $\Fbar(\gamma_\vartheta)=\Eb(\gamma_\vartheta)$, minimality of
$\gamma_\vartheta$ and \eqref{eq:rope-new-competitor} imply
\begin{equation}\label{eq:minimality-finite}
  \Fbar(\eta_\vartheta^\pm)-\Fbar(\gamma_\vartheta)
  \ge
  \vartheta\bigl(
  \Rop(\gamma_\vartheta)-\Rop(\eta_\vartheta^\pm)\bigr)
  \ge-\frac{C_h\vartheta}{\Delta_\vartheta}.
\end{equation}
After division by the positive number $\varepsilon_\vartheta$,
\eqref{eq:penalty-over-epsilon} gives
\begin{equation}\label{eq:finite-quotient-lower}
  \frac{\Fbar(\gamma_\vartheta
  \pm\varepsilon_\vartheta h)-\Fbar(\gamma_\vartheta)}
  {\varepsilon_\vartheta}
  \ge-o(1)\qquad\text{as}\quad\vartheta\downarrow0\iff\varepsilon_\vartheta\downarrow0.
\end{equation}

It remains to pass only the smooth bending part to the limit.  By the
fundamental theorem of calculus on the open set
$H^2_{\mathrm{reg}}$,
\begin{align}
  \frac{\Fbar(\gamma_\vartheta+\varepsilon_\vartheta h)
  -\Fbar(\gamma_\vartheta)}{\varepsilon_\vartheta}
  &=\int_0^1D\Fbar(\gamma_\vartheta
    +r\varepsilon_\vartheta h)[h]\,dr,
  \label{eq:FTC-plus}\\
  \frac{\Fbar(\gamma_\vartheta-\varepsilon_\vartheta h)
  -\Fbar(\gamma_\vartheta)}{\varepsilon_\vartheta}
  &=-\int_0^1D\Fbar(\gamma_\vartheta
    -r\varepsilon_\vartheta h)[h]\,dr.
  \label{eq:FTC-minus}
\end{align}
Strong $H^2$-convergence, $\varepsilon_\vartheta\to0$, and continuity
of $D\Fbar$ from \cref{lem:Fbar-smooth} show that {the right-hand sides of}
\eqref{eq:FTC-plus} and \eqref{eq:FTC-minus} converge respectively to
$D\Fbar(u)[h]$ and $-D\Fbar(u)[h]$ as $\vartheta\downarrow0$.  Passing to the limit $\vartheta\downarrow0$ in
\eqref{eq:finite-quotient-lower} yields both
$D\Fbar(u)[h]\ge0$ and $-D\Fbar(u)[h]\ge0$, proving
\eqref{eq:free-stationarity}.
\end{proof}

\begin{remark}[No ropelength derivative]\label{rem:no-rope-derivative}
The proof uses only the values
$\Rop(\eta_\vartheta^\pm)$ and the finite upper bound
\eqref{eq:rope-new-competitor}.  No directional derivative, Clarke
subgradient, balance criterion, or contact measure for ropelength is
introduced.  The only differentiated functional is the smooth local
functional $\Fbar$.
\end{remark}

\section{Classification and proof of the main theorem}
\label{sec:classification}

\begin{lemma}[Constant-curvature free-stationary curves]
\label{lem:classification}
Let $\beta\in\mathbb N$, set $\lambda:=2\pi\beta$, and let
$u\in H^2(\T,\R^3)$ be a unit-speed closed curve satisfying
\begin{equation}\label{eq:classification-assumptions}
  |u''|=\lambda\quad\text{a.e.},
  \qquad
  D\Fbar(u)[h]=0
  \quad\text{for every }h\in C^\infty(\T,\R^3).
\end{equation}
Then $u$ is, up to a Euclidean isometry and a shift or reversal of the
parameter, the $\beta$-covered circle $u_\beta$.
\end{lemma}

\begin{proof}
Put $T:=u'$ and $K:=u''$.  Formula
\eqref{eq:Fbar-first-constant} and %
\begin{equation}\label{eq:weak-K-ODE}
  0=\frac12D\Fbar(u)[h].
\end{equation}
leads to
\begin{equation}\label{eq:K-ODE}
  K''+\lambda^2K=0\qquad\text{on }\T
\end{equation}
in the sense of distributions.  Standard regularity for this linear
constant-coefficient equation (see, e.g, \cite[Thm.~6.33]{zbMATH00826154}) shows that $K$ and $u$ are smooth, and
\begin{equation}\label{eq:K-solution}
  K(s)=A\cos(\lambda s)+B\sin(\lambda s), \qquad s\in\T,
\end{equation}
for some $A,B\in\R^3$.

Since $|K|\equiv\lambda$ and the phase $\lambda s$ ranges over the full
circle, \eqref{eq:K-solution} gives
\begin{equation}\label{eq:AB-orthogonal}
  |A|=|B|=\lambda,
  \qquad A\cdot B=0.
\end{equation}
Integrating \eqref{eq:K-solution} twice yields
\begin{equation}\label{eq:u-with-drift}
  u(s)=c+vs-\frac{A}{\lambda^2}\cos(\lambda s)
  -\frac{B}{\lambda^2}\sin(\lambda s).
\end{equation}
Because $\lambda=2\pi\beta$, both trigonometric terms are one-periodic.
Closedness of $u$ therefore forces $v=0$.  Setting
$e_1:=-A/\lambda$ and $e_2:=-B/\lambda$, relation
\eqref{eq:AB-orthogonal} shows that $e_1,e_2$ are orthonormal, and
\eqref{eq:u-with-drift} becomes
\[
  u(s)=c+\frac1\lambda
  \bigl(e_1\cos(\lambda s)+e_2\sin(\lambda s)\bigr).
\]
This is the asserted $\beta$-covered circle.
\end{proof}

\begin{proof}[Proof of \cref{thm:main-intro}]
Let $\vartheta_j\downarrow0$.  By
\cref{prop:strong-convergence}, after taking a subsequence the minimizers
converge strongly in $H^2$ to a unit-speed curve $u$ with
$|u''|=2\pi\beta$ almost everywhere.  By
\cref{prop:free-stationarity}, $D\Fbar(u)[h]=0$ for every smooth variation $h$.
The conclusion follows from
\cref{lem:classification}.
\end{proof}

\begin{corollary}[Torus knots]\label{cor:torus-knots}
Let $a>b>1$ be coprime.  Under \cref{ass:GFM}, every elastic knot for
the torus knot class $T(a,b)$ is the round circle traversed $b$ times.
\end{corollary}

\begin{proof}
The bridge and braid indices of $T(a,b)$ both equal $b$; see, for
example, \cite[Section~4.1]{DiaoErnstReiter2021}.  Thus $T(a,b)$ is BB,
and \cref{thm:main-intro} applies.
\end{proof}

\bibliographystyle{amsplain}
\bibliography{elastic_bb_knots}

\end{document}